\documentclass[11pt]{article}

\usepackage{amssymb,amsmath,amsthm,amsfonts}
\usepackage[left=1in,top=1in,bottom=1in,right=1in,letterpaper]{geometry}
\usepackage{bbm}
\usepackage{hyperref}

\usepackage{titlesec}

\titleformat*{\section}{\large\bfseries}
\titleformat*{\subsection}{\bfseries}
\titleformat*{\subsubsection}{\large}
\titleformat*{\paragraph}{\large\bfseries}
\titleformat*{\subparagraph}{\large\bfseries}

\newtheorem{theorem}{Theorem}[section]
\newtheorem{proposition}[theorem]{Proposition}
\newtheorem{lemma}[theorem]{Lemma}

\theoremstyle{definition}
\newtheorem{definition}[theorem]{Definition}

\theoremstyle{remark}
\newtheorem{rmk}[theorem]{Remark}

\theoremstyle{definition}

\newcommand{\N}{\mathbb N}

\newcommand{\R}{\mathbb R}

\newcommand{\E}{\mathbb E}
\newcommand{\bbP}{\mathbb P}

\newcommand{\mcW}{\mathcal W}
\newcommand{\mcC}{\mathcal C}

\DeclareMathOperator{\supp}{supp}

\DeclareMathOperator{\ee}{e}
\DeclareMathOperator{\Leb}{Leb}

\def\la{\langle}
\def\ra{\rangle}

\newcommand{\constFromEq}[2]{#1_{\hyperref[#2]{\ref*{#2}}}}

\usepackage{color}

\title{\textbf{Compact Support Property of Super-Brownian Motion with Irregular Drift} }
\author{Leonid Mytnik, Johanna Weinberger}
\date{\today}
\begin{document}
	\maketitle
\begin{abstract}
\noindent We study the one-dimensional stochastic partial differential equation
\begin{equation*}
    d_t X_t(x)
=\frac12\Delta X_t(x)
+b_1\mathbf 1_{{X_t(x)>0}}
+\sqrt{X_t(x)}\dot W(t,x),
\end{equation*}
where $b_1>0$, $\dot W$ is space-time white noise, and the initial condition is a nonnegative, compactly supported continuous function. We prove that its unique weak solution has the compact support property. The drift term lies outside the usual regularity assumptions for the Dawson-Girsanov theorem. Instead, the proof is based on a layer decomposition in which the solution is constructed as the monotone limit of sums of super-Brownian motions with random immigration rates determined recursively by the positivity sets of the preceding layers. 

 By comparison, the compact support property extends to a broader class of bounded nonnegative drifts vanishing at the origin. 

Finally, support-radius estimates yield a comparison of the total mass of $X$ with a squared Bessel process, which allows us to show that $X$ has a positive extinction probability.

\end{abstract}	
	\section{Introduction}
	In this article, we prove the compact support property (CSP) of nonnegative solutions to the one-dimensional stochastic partial differential equation  
	\begin{equation}\label{eq:SPDE1_sup}
		d_tX_t(x) = \frac{1}{2}\Delta X_t(x) +b_1\mathbbm{1}_{X_t(x)>0}+ \sqrt{X_t(x)}\dot{W}(t,x), \quad X_0 = f, x\in \R, t\geq 0,
	\end{equation}
	where $f\in \mathcal{C}_{c}^+(\R)$, the space nonnegative continuous functions with compact support, $b_1 >0$ and $\dot{W}$ is Gaussian space-time white noise on $\R_+\times \R$. More precisely, we show that the unique weak solution to \eqref{eq:SPDE1_sup}, whose existence and uniqueness were established in \cite{Mytnik25}, is almost surely compactly supported at every time $t>0$ whenever the initial condition is compactly supported.
	Our proof relies on a construction of the unique weak solution to \eqref{eq:SPDE1_sup} as an infinite sum of super-Brownian motions with random immigration terms determined by the preceding summands. The support properties of these summands are then analyzed using classical techniques introduced in \cite{Iscoe89} together with quantitative estimates of \cite{Dawson89}, which ultimately allow us to conclude the CSP. The construction of solutions to \eqref{eq:SPDE1_sup} in this way is also of independent interest. 
	Moreover, the CSP for \eqref{eq:SPDE1_sup} also yields the compact support property for solutions to the more general equation
	\begin{equation}\label{eq:SPDE_gen}
		d_tX_t(x) = \frac{1}{2}\Delta X_t(x) +h(X_t(x))+ \sqrt{X_t(x)}\dot{W}(t,x), \quad X_0 = f, \ x\in \R,\ t\geq 0,
	\end{equation}
	provided that $h$ is a continuous function such that \(h(x)\leq b_1\mathbbm{1}_{x>0}\) for all \(x\in\R_+\) and some \(b_1>0\), and that weak uniqueness holds.
	Via a comparison with  squared Bessel processes, the estimates for the size of the support also allow us to derive estimates for the extinction probabilities of solutions to \eqref{eq:SPDE1_sup}.\\

	The compact support property for super-Brownian motion and related SPDEs has been studied by several different methods. 
	Since our approach relies on the method and results introduced by Iscoe, we briefly recall that line of work first.
	
	In the classical super-Brownian setting, Iscoe \cite{Iscoe89} proved the compact support property by combining duality with the analysis of the solutions to 
	\begin{equation} \label{eq:parabolic_infinity_1}
		\begin{split}
			&\partial_t V_t(x) = \frac{1}{2}\big( \Delta V_t(x)- V_t^2(x)\big), \quad t >0, x\in (-r,r) \\
			&V_t(x) \rightarrow \infty, \text{ as } x\rightarrow  r \text{ or } x \rightarrow -r, \\
			& V_0(x) = 0, \quad x\in [-r,r],
		\end{split}
	\end{equation}
	and 
	\begin{equation}\label{eq:elliptic_infinity_1}
		\begin{split}
			\Delta u(x) = u^2(x), \quad x\in (-r,r),\\
			u(x) \rightarrow \infty, \text{ as } x\rightarrow r \text{ or } x \rightarrow -r. 
		\end{split}
	\end{equation}
	for some $r>0$. 
	The equations \eqref{eq:parabolic_infinity_1} and \eqref{eq:elliptic_infinity_1} also play a central role in the work of Dawson, Iscoe, and Perkins \cite{Dawson89}, where further path properties of super-Brownian motion are obtained, including a modulus of continuity for the support and estimates on the probability of hitting balls.
	
	This approach was subsequently extended beyond the classical super-Brownian case. In dimension one, Shiga \cite{Shiga94} considered equations with noise coefficient of the form $a(X)\dot W$, where $a$ is a continuous function such that
	\[
    \begin{cases}
        c_1x^{1/2} \leq a(x), \quad x \in (0,K),\\
        a(x)\leq c_2(1+x), \quad x\in \R_+,
    \end{cases}
	\]
	for some $K,c_1,c_2>0$, and used an argument closely related to the ideas already present in \cite{Iscoe89}. More generally, Engl\"ander and Pinsky \cite{Englander99,Englander06} developed an extension of Iscoe's method to a broad class of measure-valued diffusions.
	Recent works also show that the approach from \cite{Iscoe89} remains useful in singular settings. In particular, Jin and Perkowski \cite{Jin25} established the compact support property for rough super-Brownian motion on $\mathbb{R}^2$ by combining the classical measure-valued-process approach of \cite{Englander99,Englander06} with the interior estimate method developed in \cite{Moinat20}.

	An alternative approach to analyzing stochastic heat equations with multiplicative noise of the form $X^\gamma \dot W$ is through the use of historical processes.  Mueller and Perkins \cite{Mueller92} established compact support for $\gamma\in(0,1)$ using this theory. Their method also yields additional information on the support; see, for example, \cite[Corollary 3.8]{Mueller92}. In dimension one this range for $\gamma$ is sharp: it was proven in \cite{Mueller91} that the compact support property fails when $\gamma \geq 1$.
	
	Later, Krylov \cite{Krylov97} gave another proof for the compact support property in one dimension for SPDEs with multiplicative noise of the form
	\begin{equation*}
		d_tX_t(x) = a(t,x)\Delta X_t(x) + b(t,x) \partial_xX_t(x) +c(t,x)X_t(x) + v(t,x)\big(X_t(x)\big)^\gamma \dot{W}(t,x), \quad t>0, x\in \R, 
	\end{equation*}
	where $a,b,c,v$ are possibly random, measurable functions satisfying certain conditions and $\gamma \in (0,1)$.
	This approach was inspired by methods used for PDEs, in particular the maximum principle.
	This approach has proven flexible enough to treat more general driving noises. Han, Kim, and Yi \cite{Han23} extended this method to stochastic PDEs driven by spatially colored noise, which allows one to consider higher spatial dimensions. More recently, Hughes \cite{Hughes25} adapted the same strategy to equations driven by $\alpha$-stable noise, proving compact support in dimension one for noise coefficients of the form $X^\gamma \dot L$ with $\gamma\in(2-\alpha,1)$. 
	
	In the case of super-Brownian motion with drift we note that in some situations the compact support property  may also be obtained via a Dawson--Girsanov transform, as demonstrated in \cite{Muller94}. However, this approach is not available in our setting, since the hypotheses of the Dawson--Girsanov theorem fail because of the irregularity of the drift. For this reason we introduce a different approach for proving the compact support property of solutions to \eqref{eq:SPDE1_sup}, which relies on the representation of the solution as an infinite series. To the best of our knowledge, this is the first result establishing the CSP of super-Brownian motion with drifts $h$ such that $h(x) \geq cx^\beta$ for $\beta\in (0,1/2)$ and $x$ close to zero or where $h$ has a jump discontinuity at $0$.
	
	The compact support property is closely related to the extinction probability of solutions to \eqref{eq:SPDE_gen}. This was, for instance, demonstrated in \cite{Muller94}, where the authors used estimates for the size of the support of the solution to \eqref{eq:SPDE_gen} with $h(x) = \theta x- x^2$ to compare the total mass of the solution with a certain SDE. In this article we will use a similar approach to estimate the extinction probability of solutions to \eqref{eq:SPDE1_sup}. 
    \paragraph{Organization}
    In Section~\ref{sec:not}, we introduce notation used throughout the article. \\~\\
    In Section~\ref{sec:main_res}, we state the main results of the article. \\~\\
    Next, in Section~\ref{sec:initial_boundary}, we recall several results about initial and boundary trace problems for the log-Laplace equation.\\~\\
    In Section~\ref{sec:construction_seq}, we construct the sequence of super-Brownian motions with immigration that we later use to derive a new representation of the weak solution to \eqref{eq:SPDE1_sup}.\\~\\
    In Section~\ref{sec:cozero}, we derive estimates for the cozero sets of the processes constructed in Section~\ref{sec:construction_seq}.\\~\\
    In Section~\ref{sec:radius}, we derive estimates for the maximal radius of the support of the partial sums of the processes constructed in Section~\ref{sec:construction_seq}.\\~\\
    In Section~\ref{sec:series_rep}, we prove Theorem~\ref{th:convergence_X_bar_n}. \\~\\
    In Section~\ref{sec:compact_support}, we prove Proposition~\ref{prop:expectation_RT} and Theorem~\ref{th:compact_support}.\\~\\
    Finally, in Section~\ref{sec:extinction}, we provide the proof of Theorem~\ref{prop:extinction}.
    \paragraph{Acknowledgments}
    The work of the authors was supported in part by the ISF grant No. 1985/22.

    \paragraph{Statement on AI use}
    The mathematical results and core arguments of this paper were developed by the authors independently. At a late stage of preparation, the authors used OpenAI’s ChatGPT [5.6 Pro; accessed August 2026] for critical reading and language editing of the manuscript, including identifying passages and proof steps requiring further clarification. All resulting revisions were independently evaluated, implemented, and verified by the authors, who assume full responsibility for the content.

	\subsection{Notation}\label{sec:not}
	We use the following notation for common function spaces.
	\begin{enumerate}
		\item $\mathcal{C}(\R)$ denotes the space of real-valued, continuous functions on $\R$ and $\mathcal{C}^k(\R)$, $k \in N$ denotes the space of $k$-times continuously differentiable, real-valued functions on $\R$. Furthermore, we denote $\mcC_b^k(\R) = \{ f \in \mcC^k(\R) \colon f, f', \ldots f^{(k)} \text{ are bounded}\}$.
		\item $\mathcal{C}_c(\R)$ denotes the space of compactly supported functions in $\mathcal{C}(\R)$. 
		\item $	\mathcal{C}_{tem} = \{ f \in \mathcal{C}(\R) \colon |f|_{(\lambda)}<\infty \quad \forall \lambda <0\}$, where $	|f|_{(\lambda)} = \sup_{x\in \R}|\ee^{\lambda|x|}f(x)|$ for $f\in \mathcal{C}(\R)$ and we equip this space with the topology induced by the metric
            \[
                d_{\mathrm{tem}}(f,g):=\sum_{n=1}^{\infty}2^{-n}\bigl(1\wedge |f-g|_{(-1/n)}\bigr).\]
		\item $\mathcal{C}(\R_+,\mathcal{C}_{tem}^+)$ is the space of functions $f\colon \R_+ \rightarrow \mathcal{C}_{tem}^+$, where continuity is understood with respect to \(d_{\mathrm{tem}}\).
		\item $L^p(A,m)$ for $p \in [1,\infty]$ denotes the usual $L^p$-space on the measure space $(A, \mathcal{A},m)$. If there is no ambiguity we write $L^p$ instead of $L^p(A,m)$ for $p \in [1,\infty]$.
		\item $\mathcal{M}_F^+$ denotes the set of finite, nonnegative measures on $\R$.
		\item We denote by $(\mathcal{M}_F^+, d_{bl})$ the metric space of nonnegative, finite measures on $\R$ with the bounded Lipschitz metric $d_{bl}$ defined as 
		\begin{equation*}
			d_{bl}(\mu,\nu) = \sup_{f\in \mathcal{L}} \Big| \langle \mu ,f \rangle - \langle \nu , f\rangle\Big|,  
		\end{equation*}
		where 
		\begin{equation*}
			\mathcal{L} = \{ f \in \mcC(\R) \colon \|f\|_\infty \leq 1, \|f\|_{Lip} \leq 1\}
		\end{equation*}
		and $\|f\|_{Lip} = \sup_{x\neq y} \frac{|f(x)-f(y)|}{|x-y|}$. By straightforward modifications, it follows from \cite[Theorem 8.3.2]{bogachev2007measure} that $d_{bl}$ induces the topology of weak convergence. 
		\item When $ \mu \in \mathcal{M}_F^+$ has a density $f$ with respect to the Lebesgue measure, we sometimes use the notations $fdx$ and $f$ interchangeably, in a slight abuse of notation. 
		\item $L^1(\Omega, \mcC([0,T], \mathcal{M}_F^+)) = L^1(\Omega, \mcC([0,T], (\mathcal{M}_F^+, d_{bl})))$  is the metric space (equivalence classes) of measurable functions $f\colon \Omega \rightarrow  \mcC([0,T], \mathcal{M}_F^+)$ so that 
		\begin{equation*}
			 \E[\| \langle f, 1 \rangle\|_\infty] <\infty,
		\end{equation*}
		with the metric 
		\begin{equation*}
			d(f,g) = 	 \E[\| d_{bl}(f,g)\|_\infty].
		\end{equation*}
		Note that $ \mcC([0,T], \mathcal{M}_F^+)$ is a complete, separable metric space and thus the notions of measurability and strong measurability coincide by Pettis theorem. Furthermore, it can be shown that  $L^1(\Omega, \mcC([0,T], \mathcal{M}_F^+))$ is complete along the same lines as the usual proof of completeness for Lebesgue spaces. 
		\item $\mathcal{B}(\R)$ denotes the Borel $\sigma$-algebra.
	\end{enumerate}
	If $A$ is a subspace of $\mathcal{C}(\R)$, we denote by $A^+$ the subset of nonnegative functions in $A$. Furthermore, we denote by $A_c$ the subspace of compactly supported functions in $A$. 
	We also introduce the following notation
	\begin{equation*}
		\la \mu,f\ra = \int_\R f(x)\mu(dx), 
	\end{equation*}
	for a measure $\mu \in \mathcal{M}_F^+$ and a function $f\in \mathcal{C}_b(\R)$. Furthermore, if $g\in L^1(\R)$ and $f\in \mcC_b(\R)$ we also write 
	\begin{equation*}
		\la f,g\ra = \int_\R f(x)g(x)dx.
	\end{equation*}
	
	\section{Main Results}\label{sec:main_res}
	In this section, we introduce the main results of this article. To formulate them precisely, we first specify the notion of a weak solution to \eqref{eq:SPDE1_sup} used throughout.

    \begin{definition}
			We say that \eqref{eq:SPDE1_sup} has a weak solution $X_t$, if there exists a filtered probability space $(\Omega, \mathcal{F}, (\mathcal{F}_{t})_{t\geq 0}, \bbP)$ such that $(X_t)_{t\geq 0}$ is an adapted $\big(\mathcal{C}_{tem}^+\big)$-valued 
		continuous process with $X_0 =f $ and such that $\dot{W}$ is space-time white noise and for every $(t,x) \in (0,\infty)\times \R$ almost surely
		\begin{equation*}
			\begin{split}
				&X_t(x) = \int_\R p_t(x-y) f(y)dy + \int_\R \int_0^t p_{t-s}(x-y)b_1 \mathbbm{1}_{X_s(y)>0}dsdy \\
				& \hspace{2cm}+ \int_\R \int_0^t p_{t-s}(x-y) \sqrt{X_s(y)}W(dsdy),\quad \bbP-a.s.,
			\end{split}
		\end{equation*}
        where 
        \begin{equation*}
        p_t(x) = \frac{1}{\sqrt{2 \pi t}}\exp(-x^2/(2t)), \quad t>0, x\in\R.
        \end{equation*}
	\end{definition}
	Before proving the compact support property for \eqref{eq:SPDE1_sup}, we first derive a particular construction of its unique weak solution.
To this end, we first introduce the following sequence.
	Let $X^0$ be the unique solution to 
	\begin{equation}\label{eq:SPDE_0}
		d_tX_t^0(x) = \frac{1}{2}\Delta X_t^0(x)+ \sqrt{X_t^0(x)}\dot{W^0}(t,x), \quad X_0^0 = f, x\in \R, t\geq 0,
	\end{equation}
	where $\dot{W}^0$ is space-time white noise and $f\in \mcC_c^+(\R)$. 
	Now, we recursively construct processes $X^k$ for $k\geq 1$. To this end, denote by $(\mathcal{F}^{k-1}_t)_{t\geq 0}$ the completed, right-continuous natural filtration generated by $(X^0,\ldots X^{k-1})$ for $k\geq 1$ and set
    \begin{equation*}
		\mathcal{F}_\infty^{k-1} = \sigma\Big(\cup_{t>0}\mathcal{F}^{k-1}_t\Big).
	\end{equation*}
	Then,  conditional on $\mathcal{F}_\infty^{k-1}$, we construct $X^k$ as the solution to 
	\begin{equation}\label{eq:SPDE_k}
		d_tX_t^k(x) = \frac{1}{2}\Delta X_t^k(x) +g^k(t,x)+ \sqrt{X_t^k(x)}\dot{W^k}(t,x), \quad X_0^k \equiv0,  x\in \R, t\geq 0,
	\end{equation}
	where $\{\dot{W}^k\}_{k\geq 0}$ are independent space-time white noises  and 
	\begin{equation}\label{eq:def_g_k}
		g^k(t,x) = b_1\mathbbm{1}_{X^{k-1}_t(x) >0, \sum_{i =1}^{k-1}g^i(t,x) = 0}, \quad x\in \R, t \geq 0 
	\end{equation}
    for $k\geq 2$ and 
    \begin{equation*}
        g^1(t,x) = b_1\mathbbm{1}_{X^{0}_t(x) >0}, \quad x\in \R, t \geq 0.
    \end{equation*}
	Note that $g^{k}$ is measurable with respect to $\mathcal{F}^{k-1}_\infty$ and that $\dot{W}^k$ is independent of $\mathcal{F}^{k-1}_\infty$. Conditionally on $\mathcal{F}_\infty^{k-1}$, $X^k$ is a superprocess with immigration given by $g^k$, which is positive at points $(t,x)\in \R_+\times \R$ where $X^{k-1}$ is positive but $X^i_t(x) = 0$ for $i \in \{0,\ldots k-2\}$ when $k\geq 2$.\par
	We furthermore define the sequence 
	\begin{equation}\label{eq:def_X_bar_n}
		\bar{X}_t^n(x) := \sum_{k =0}^nX^k_t(x), \quad t\geq0, x\in \R.
	\end{equation}
	Using standard arguments, it is possible to show that, on an enlarged probability space, the process $\bar{X}^n$ satisfies
	\begin{equation*}
		d_t\bar{X}^n_t(x) = \frac{1}{2}\Delta \bar{X}_t^n + b_1\mathbbm{1}_{\bar{X}_t^{n-1}(x)>0} + \sqrt{\bar{X}^n}\dot{\tilde{W}}^n(t,x), \quad \bar{X}_0^n = f, x\in \R, t\geq 0,
	\end{equation*}
	where $\dot{\tilde{W}}^n$ is space-time white noise.
	Weak existence and uniqueness of solutions to \eqref{eq:SPDE_k} are proved in Section~\ref{sec:construction_seq}, which allows us to construct \eqref{eq:def_X_bar_n}. \\
	In the next theorem, we show that the sequence \eqref{eq:def_X_bar_n}  converges to the unique weak solution to \eqref{eq:SPDE1_sup}. Here, the convergence is in the space 
    \begin{equation*}
        L^1(\Omega; \mcC([0,T], (\mathcal{M}_F^+, d_{bl}))),
    \end{equation*}
    where $d_{b_l}$ denotes the bounded Lipschitz metric defined as 
		\begin{equation*}
			d_{bl}(\mu,\nu) = \sup_{f\in \mathcal{L}} \Big| \langle \mu ,f \rangle - \langle \nu , f\rangle\Big|,  
		\end{equation*}
		where 
		\begin{equation*}
			\mathcal{L} = \{ f \in \mcC(\R) \colon \|f\|_\infty \leq 1, \|f\|_{Lip} \leq 1\}
		\end{equation*}
		and $\|f\|_{Lip} = \sup_{x\neq y} \frac{|f(x)-f(y)|}{|x-y|}$. By \cite[Theorem 8.3.2]{bogachev2007measure} it follows that $d_{bl}$ induces the topology of weak convergence. 
	\begin{theorem}\label{th:convergence_X_bar_n}
		There exists a filtered probability space $(\Omega, (\mathcal{F}_t)_{t\geq 0}, \bbP)$ such that for any $T >0$ the sequence $(\bar{X}^n)_{n\geq 1}$ converges in $L^1(\Omega;\mcC([0,T], (\mathcal{M}_F^+, d_{bl})))$ to the unique weak solution $X$ of \eqref{eq:SPDE1_sup} in $\mcC([0,T], \mcC_{tem}^+)$.
	\end{theorem}
	We prove the above theorem in Section~\ref{sec:series_rep}.
	The above construction allows us  not only to establish the compact support property for \eqref{eq:SPDE1_sup}, but also to estimate the expected value of the radius of the maximal support of the solution defined as
	\begin{equation*}
		R_T = \sup\{|x| \colon \exists t\in [0,T] \text{ such that } X_t(x) >0 \},
	\end{equation*}
	for a time $T \geq0$. If we want to emphasize $R_T$'s  dependence on $X_0 = f$ we will write $R_T^f$ and we use the convention $\sup \emptyset = 0$.
	\begin{proposition}\label{prop:expectation_RT}
		Let $b_1>0$ and let $X$ be the unique weak solution to \eqref{eq:SPDE1_sup} in $\mcC(\R_+, \mcC_{tem}^+)$ with $X_0 \equiv f\in C_c^+(\R)$ and denote $m_f = \langle f,1 \rangle$ and
        \begin{equation*}
            R_0 =  \sup\{|x| \colon f(x) >0 \},
        \end{equation*}
        with the convention $\sup \emptyset =0$.
        Then there exists a finite constant $\beta(T, m_f) \geq 0$  for every $T>0$ such that $\beta(T,m_f)\rightarrow 0$ as $m_f \downarrow 0$  and 
		\begin{equation}\label{eq:expectation_RT_est}
			\E_f[R_T] \leq A(T)(R_0 +\sqrt{R_0} + \beta(T,m_f)), 
		\end{equation}
		where  $A(T) = 1+ \sqrt{2\pi b_1 \constFromEq{c}{eq:est_v_infty_1}}T^{1/2} + \sum_{i =1}^\infty a_i(T) <\infty$ with $a_i$ given by \eqref{eq:def_ai}.
	\end{proposition}
	The above proposition immediately implies the following theorem, which is the central result of this article.  
	\begin{theorem}\label{th:compact_support}
       Let $b_1>0$ and $X$ denote the unique weak solution to \eqref{eq:SPDE1_sup} in $\mcC(\R_+, \mcC_{tem}^+)$ with $X_0 \equiv f\in\mcC_c^+(\R)$, then 
       \begin{equation*}
             \bbP_f\bigl(X_t\in \mcC_c^+(\R), \quad \forall t \geq 0\bigr) = 1.
        \end{equation*}
	\end{theorem}
	\noindent We prove Proposition~\ref{prop:expectation_RT} and Theorem~\ref{th:compact_support} in Section~\ref{sec:compact_support}.
    \textcolor{black}{
    \begin{rmk}\label{rem:spde_gen}
        Let $h\colon \R_+ \rightarrow \R_+$ be given by 
        \begin{equation*}
            h(x) = b_1\mathbbm{1}_{x>0} + \int_0^\infty (1-\ee^{-\lambda x})\nu(d\lambda), \quad x\in \R_+,
        \end{equation*}
    where $\nu$ is a finite measure on $(0,\infty)$ and $f\in \mcC^+_c(\R)$. By \cite[Theorem 1.1]{Mytnik25} there exists a unique weak solution to 
    \begin{equation}\label{eq:spde_gen}
    d_tX_t(x) = \frac{1}{2}\Delta X_t(x) +h(X_t(x))+ \sqrt{X_t(x)}\dot{W}(t,x), \quad X_0 = f, x\in \R, t\geq 0,   
    \end{equation}
    in $\mcC(\R_+, \mcC_{tem}^+)$.
    By a standard approximation argument, weak convergence and weak uniqueness established in \cite{Mytnik25} and a comparison theorem (e.g. \cite[Corollary 2.4]{Shiga94}),
    there exists a probability space such that
    \begin{equation*}
       \bbP(\hat{X}_t(x) \geq X_t(x),  \forall t\geq 0, x\in \R)=1,
    \end{equation*}
    where $\hat{X}$ is the unique weak solution to \eqref{eq:SPDE1_sup} with $b_1$ replaced by $\hat{b}_1 = b_1 + \nu((0,\infty))$. Thus, Theorem~\ref{th:compact_support} implies the compact support property for \eqref{eq:spde_gen}.
    \end{rmk}}

 As mentioned above, estimates for $R_T$ allow us to estimate the extinction probability of solutions to \eqref{eq:SPDE1_sup}.
 In particular, for $T>0$, Proposition~\ref{prop:expectation_RT} allows us to compare the total mass of the solution \(X\) on $[0,T]$ with the  solution to
	\begin{equation}\label{eq:sde_compare}
		\tilde{z}_t = \langle f,1\rangle + \delta t + \int_0^t\sqrt{\tilde{z}_s}dB_s, \quad t\in [0,T],
	\end{equation}
	where $(B_t)_{t\geq 0}$ is a standard Brownian motion and $\delta$ is chosen so that $2b_1R_T<\delta$ with high probability. It turns out that, if $T$, $\langle f,1\rangle $ and the diameter of the support of  $f$ are sufficiently small, the parameter $\delta>0$ can be chosen small enough so that $\tilde{z}$ hits zero with positive probability before time $T$. A careful analysis thus yields the following result. 
	\begin{theorem}\label{prop:extinction}
	Let $b_1>0$ and $X$ be the unique weak solution to \eqref{eq:SPDE1_sup} in $\mcC(\R_+, \mcC_{tem}^+)$ with $X_0 \equiv f\in\mcC_c^+(\R)$. 
	Then
		\begin{equation}\label{eq:extinction}
			\mathbbm{P}_f(\langle X_t, 1\rangle = 0) >0, 
		\end{equation}
	for all $t>0$.
	
	\end{theorem}
	\noindent The proof of the above theorem is deferred to Section~\ref{sec:extinction}. 
    \textcolor{black}{
    \begin{rmk}
        As in Remark~\ref{rem:spde_gen} we can apply standard approximation and weak convergence arguments, weak uniqueness and a comparison theorem to prove that \eqref{eq:extinction} also holds for solutions to \eqref{eq:spde_gen}.
    \end{rmk}}
    
\section{Initial and Boundary Trace Problems}\label{sec:initial_boundary}
In this section, we consider equations of the form
\begin{equation}\label{eq:PDE_mu_psi}
	\begin{split}
		&\partial_t V_t(x) = \frac{1}{2}\Delta V_t(x) - \frac{1}{2}(V_t(x))^2 + \psi(x), \quad t >0, x\in \R\\
		&V_0= \mu,
	\end{split}
\end{equation}
with possibly very irregular $\psi$ or initial data $\mu$. If we want to make the dependence on $\psi$ and $\mu$ clear, we write $V(\mu,\psi)$, and if there is no ambiguity, we may just write $V$, $V(\mu)$ or $V(\psi)$.
This partial differential equation plays a central role in our proofs due to its close relation
with the dual process of super-Brownian motion. 
\begin{proposition}\label{prop:existence}
	Let $\psi, \mu \in \mathcal{C}_c^+(\R)$. Then there exists a unique mild solution to \eqref{eq:PDE_mu_psi}, which is nonnegative.
\end{proposition}
\subsection{Initial Trace Problems}
    In this section, we discuss the case where $\mu$ is irregular. That is, we consider the equation
\begin{equation}\label{eq:parabolic_eq}
	\partial_t V_t = \frac{1}{2}\Delta V_t - \frac{1}{2}V_t^2, \quad t\geq 0
\end{equation}
on $\R$ with irregular initial data. 
For a more thorough discussion, we refer the reader to \cite{Marcus99} and \cite{LeGall96}. Here, we give only  a brief overview.
More precisely, we discuss some basic properties of solutions to 
\begin{equation}\label{eq:PDE}
	\begin{cases}
		&\partial_t V_t(\mu) = \frac{1}{2}\Delta V_t(\mu) - \frac{1}{2}V_t(\mu)^2, \quad (t,x ) \in (0,\infty) \times \R \\
		& V_t(\mu) \Rightarrow \mu, \text{ as } t \downarrow 0,
	\end{cases}
\end{equation}
where $\mu$ is a finite measure, as well as solutions to the initial trace problem
\begin{equation}\label{eq:initial_trace}
	\begin{cases}
		&\partial_t \mathcal{W}_t^{x_0} = \frac{1}{2}\Delta \mathcal{W}_t^{x_0} - \frac{1}{2}(\mathcal{W}_t^{x_0})^2, \quad (t,x ) \in (0,\infty) \times \R\\
		& \lim_{t \rightarrow 0}\int_{x_0-r}^{x_0+r} \mathcal{W}_t^{x_0}(x)dx = \infty, \forall r >0\\
		& \lim_{t \rightarrow 0} \int \phi \mathcal{W}_t^{x_0}(x) dx = 0, \forall \phi \in \mathcal{C}_c(\R\backslash\{x_0\}), 
	\end{cases}
\end{equation}

First, we get existence and uniqueness for \eqref{eq:PDE} by \cite[Lemma 2.1]{Mytnik02}, which we restate in the following proposition.
\begin{proposition}
	Let $\mu$ be a finite measure. Then there exists a unique solution $V(\mu)$ to \eqref{eq:PDE} in  $L^1_{loc}((0,\infty) \times \R)$ satisfying
	\begin{equation}\label{eq:sol_rep_V}
		V_t(\mu)(x) = (S_t\mu)(x) -\int_0^t \frac{1}{2}\Big(S_{t-s}V_s(\mu)^2\Big)(x)ds, \quad t>0, x\in \R,
	\end{equation}
	where $(S_t)_{t\geq 0}$ is the heat semigroup with transition density $(p_t)_{t \geq 0}$ given by 
    \begin{equation*}
         p_t(x) = \frac{1}{\sqrt{2 \pi t}}\exp(-x^2/(2t)), \quad t>0, x\in\R.
    \end{equation*}
   
\end{proposition}

The following lemma is a well-known fact about \eqref{eq:PDE} (see, for instance, \cite[Lemma 2.1]{Mytnik02}).
\begin{lemma}\label{lem:est_pde}
	Let $\mu$ be a finite measure. Then it holds that 
	\begin{equation*}
		V_t(\mu)(x) \leq S_t\mu(x),
	\end{equation*}
	for all $t > 0$ and $x\in \R$.
\end{lemma}
The following lemma, taken from \cite[Lemma 2.6]{Mytnik02}, provides estimates that we use frequently. 
\begin{lemma}\label{lem:pde_est}
	Let $\mu$ and $\eta$ be two finite measures. Then the following holds:
	\begin{enumerate}
		\item[(i)]
		For all $t > 0$ it holds that
		\begin{equation*}
			V_t(\mu+ \eta)(x) \leq V_t(\mu)(x) + V_t (\eta)(x), \quad x\in\R.
		\end{equation*}
		\item[(ii)] If $\mu \leq \eta$, then 
		\begin{equation*}
			V_t(\mu)(x) \leq V_t(\eta)(x),\quad x\in \R,
		\end{equation*}
		for all $t > 0$.
	\end{enumerate}
\end{lemma}
The following proposition is a version of  Theorem~1 from \cite{Brezis95}.
\begin{proposition}\label{prp:very_sing_asymp}
	For any $x_0\in\R$, there exists a unique nonnegative solution to \eqref{eq:initial_trace} such that 
	\begin{equation}\label{eq:def_W}
		 \mathcal{W}_t^{x_0}(x) = t^{-1} \xi(t^{-1/2} |x-x_0|),
	\end{equation}
	where $\xi \colon [0,\infty) \rightarrow [0,\infty)$ is a smooth function, independent of $x_0$, satisfying
	\begin{equation}\label{eq:est_f}
		\xi(x) = C\ee^{-\frac{ x^2}{2}}x\Big( 1+ o(x^{-2})\Big), 
	\end{equation}
	as $x \rightarrow \infty$. Also, there exists a constant $\constFromEq{c}{eq:est_W_int}>0$ so that 
		\begin{equation}\label{eq:est_W_int}
			\int_\R \mathcal{W}^{x_0}_t(x)dx \leq \constFromEq{c}{eq:est_W_int}t^{-1/2}, \quad t>0,
		\end{equation}
    for every $x_0\in\R$.
\end{proposition}

In particular, by Lemma~\ref{lem:pde_est} and \cite[Corollary 3, Lemma 4]{Kamin95} it follows that $V_t(n\delta_x)(y)$ converges monotonically to $\mathcal{W}_t^x(y)$ for all $t > 0$
and $x,y\in \R$. Therefore, 
\begin{equation}\label{eq:VW_ineq}
	V_t(n\delta_x)(y) \leq \mathcal{W}_t^x(y), \forall t > 0, x,y\in\R, \quad n \geq 1,
\end{equation}
which we shall frequently use in subsequent sections.
\subsection{Boundary Trace Problems}
In this section, we review the case where $\psi$ may be infinite on sets of positive Lebesgue measure. 
The section is based on results from \cite{Iscoe86, Iscoe89,Dawson89}.\\
The following proposition is Theorem 3.3 from \cite{Iscoe86}. 

	\begin{proposition}\label{prop:monotonicity_t}
		Let $\mu \equiv 0$ and $\psi \in \mcC_c^+(\R)$. Then the solution $V$ to \eqref{eq:PDE_mu_psi} is nondecreasing in $t$ and, as $t\rightarrow\infty$,  converges uniformly on compacts to the unique mild solution to 
		\begin{equation*}
			\frac{1}{2}\Delta u(x) -\frac{1}{2}(u(x))^2 + \psi(x ) = 0, \quad x \in \R. 
		\end{equation*}
	\end{proposition}
	
	\begin{lemma}\label{lem:elliptic}
		Let $r >0$, $m >0$, and 
		\begin{equation}
			\psi^{m,r}(x) =
			\begin{cases}
				0, & |x| \leq r, \\
				m(|x|/r - 1), & r \leq |x| \leq 2r, \\
				m, & 2r \leq |x|.
			\end{cases}
		\end{equation}
		Then there exists a unique mild nonnegative solution $u^m$ to 
		\begin{equation}\label{eq:elliptic_m}
			 \frac{1}{2}\Delta u^m(x) -\frac{1}{2}(u^m(x))^2 + \psi^{m,r}(x ) = 0, \quad x \in \R. 
		\end{equation}
		Furthermore, $u^m(x)$ is monotonically increasing in $m$ for all $x\in R$ and there exists a constant $\constFromEq{c}{eq:est_u_infty}>0$ such that for $|x| < r$ we have
		\begin{equation}\label{eq:est_u_infty}
			\lim_{m \rightarrow \infty} u^m(x) \leq \constFromEq{c}{eq:est_u_infty}(r-|x|)^{-2}.
		\end{equation}
    Furthermore, the monotone pointwise limit $u$ of $u^m$ is the unique, nonnegative solution to
    \begin{equation}\label{eq:boundary_trace_elliptic}
        \begin{cases}
            \Delta u(x) = \big( u(x) \big)^2, \quad x\in (-r,r), \\
            u(x) \rightarrow \infty, \text{ as } x\rightarrow r \text{ or } x\rightarrow -r.
        \end{cases}
    \end{equation}
	\end{lemma}
	\begin{proof}
		Existence and uniqueness of the mild solution to \eqref{eq:elliptic_m} follow along the same lines as the proof of \cite[Lemma 3.4]{Iscoe89}. The estimate \eqref{eq:est_u_infty} follows from \cite[Lemma 3.6]{Dawson89}. Finally, it follows from \cite[Proposition 3.5]{Iscoe89} that $u$ satisfies \eqref{eq:boundary_trace_elliptic}.
	\end{proof}
	\begin{lemma}
		Let $r >0$, $m >0$, and 
		\begin{equation}
			\psi^{m,r}(x) =
			\begin{cases}
				0, & |x| \leq r, \\
				m(|x|/r - 1), & r \leq |x| \leq 2r, \\
				m, & 2r \leq |x|.
			\end{cases}
		\end{equation}
		Then there exists a unique mild solution $V(0,\psi^{m,r})=V^{m,r}$ to 
		\begin{equation}\label{eq:parabolic_m}
			\begin{cases}
				&\partial_t V^{m,r}_t(x) = \frac{1}{2}\Delta V^{m,r}_t(x) - \frac{1}{2}(V^{m,r}_t(x))^2 + \psi^{m,r}(x), \quad t >0, x\in \R,\\
				&V_0^{m,r}\equiv 0\\
				& V_t^{m,r}(x) \geq 0,  \quad t >0, x\in \R.
			\end{cases}
		\end{equation}
		Furthermore, for $(t,x) \in \R_+\times \R$, the sequence $(V_t^{m,r}(x))_{m> 0}$ is monotonically increasing in $m$.  We get the bounds
		\begin{equation}\label{eq:est_v_infty_1}
			\lim_{m \rightarrow \infty}V_t^{m,r}(x) \leq \constFromEq{c}{eq:est_v_infty_1} (r-|x|)^{-2}, \quad t > 0,  -r<x<r,
		\end{equation}
		and
		\begin{equation}\label{eq:ext_v_infty_2}
			\lim_{m \rightarrow \infty}V_t^{m,r}(x) \leq \constFromEq{c}{eq:ext_v_infty_2} (r-|x|) t^{-\frac{3}{2}} \exp\Big(-\frac{(r-|x|)^2}{2t}\Big), \quad t>0, r-|x|>2\sqrt{t},
		\end{equation} 
		for some constants $\constFromEq{c}{eq:est_v_infty_1}, \constFromEq{c}{eq:ext_v_infty_2}>0$.
	\end{lemma}
	\begin{proof}
		Existence of $V^{m,r}$ for $m,r>0$ follows by approximating $\psi^{m,r}$ with compactly supported functions and using the monotone convergence theorem.
		 Similarly, an application of Proposition~\ref{prop:monotonicity_t} and Lemma~\ref{lem:elliptic} also yields \eqref{eq:est_v_infty_1}. It remains to verify \eqref{eq:ext_v_infty_2}. To this end, we introduce an additional probability space $(\tilde{\Omega}, \tilde{\mathcal{F}}, \tilde{\bbP})$ satisfying the usual conditions and supporting a standard Brownian motion $(B_t)_{t\geq 0}$.  
		Following arguments similar to those in \cite[Theorem 3.3]{Dawson89} we obtain
		\begin{equation*}
			\lim_{m \rightarrow \infty}V_t^{m,r}(x) \leq u(\tilde{r}, r) \tilde{\bbP}_x(T_{\tilde{r}}<t),
		\end{equation*} 
		for $|x|<\tilde{r} <r$, where $T_{\tilde{r}} := \inf\{t >0 \colon |B_t| > \tilde{r}\}$, under $\tilde{P}_x$ $B_0 = x$  and $u = \lim_{m\rightarrow \infty}u^m$  with $u^m$ denoting the solution to \eqref{eq:elliptic_m}. Again, as in \cite{Dawson89}, we show the existence of a constant $\constFromEq{c}{eq:ext_v_infty_2}>0$ such that 
		\begin{equation*}
			\lim_{m \rightarrow \infty}V_t^{m,r}(x) \leq \constFromEq{c}{eq:ext_v_infty_2} (r-|x|) t^{-\frac{3}{2}} \exp\Big(-\frac{(r-|x|)^2}{2t}\Big), \quad r-|x|>2\sqrt{t}.
		\end{equation*} 
	\end{proof}
	\section{A Sequence of Superprocesses }\label{sec:construction_seq}
	In this section, we establish the existence and uniqueness of solutions to \eqref{eq:SPDE_k}. \\ 
	To this end, let us consider the following system with more general initial conditions.  For the remainder of this section, let $X^0$ be the unique weak solution to 
	\begin{equation}\label{eq:SPDE_0_gen}
		d_tX_t^0(x) = \frac{1}{2}\Delta X_t^0(x)+ \sqrt{X_t^0(x)}\dot{W^0}(t,x), \quad X_0^0 = f^0, x\in \R, t\geq 0,
	\end{equation}
	where $\dot{W}^0$ is space-time white noise and $f^0 \in \mathcal{C}_{tem}^+$. Recall that the unique weak solution exists by \cite[Corollary III.4.3]{Perkins02}.
	Now, we recursively construct processes $X^k$ for $k\geq 1$ conditional on $\mathcal{F}^{k-1}_\infty$, which is the $\sigma$-algebra generated by $(X^0, \ldots X^{k-1})$, as the solutions to 
	\begin{equation}\label{eq:SPDE_k_gen}
		d_tX_t^k(x) = \frac{1}{2}\Delta X_t^k(x) +g^k(t,x)+ \sqrt{X_t^k(x)}\dot{W^k}(t,x), \quad X_0^k = f^k, x\in \R, t\geq 0,
	\end{equation}
	where $\{\dot{W}^k\}_{k\geq 0}$ are independent space-time white noises, $f^k \in \mcC_{tem}^+$  and $g_k$ is defined as in \eqref{eq:def_g_k}.
	For any $n\geq 1$, let us first prove weak existence of solutions to 
		\begin{equation} \label{eq:SPDE_n}
		\begin{cases}
			& d_tX_t^0(x) = \frac{1}{2}\Delta X_t^0(x)+ \sqrt{X_t^0(x)}\dot{W^0}(t,x), x\in \R, t\geq 0,\\
			&d_tX_t^k(x) = \frac{1}{2}\Delta X_t^k(x)+g^k(t,x)+ \sqrt{X_t^k(x)}\dot{W^k}(t,x),  x\in \R, t\geq 0, 1\leq k \leq n,
		\end{cases}
	\end{equation}
	with initial condition $(X^0_0, \ldots X^n_0)\in \big(\mcC_{tem}^+\big)^{n+1}$. 
	Hereby,  we use the following notion of solutions. 
	\begin{definition}
			We say that \eqref{eq:SPDE_n} has a weak solution $(X_t^0,\ldots, X_t^n )_{t\geq 0}$, if there exists a filtered probability space $(\Omega, \mathcal{F}, (\mathcal{F}_{t})_{t\geq 0}, \bbP)$ such that $(X_t^0,\ldots, X^n_t)_{t\geq 0}$ is an adapted $\big(\mathcal{C}_{tem}^+\big)^{n+1}$-valued 
		continuous process with $X_0^k =f^k $ for $0\leq k \leq n$, and such that $(\dot{W}^k)_{0\leq k \leq n}$  is a family of independent  space-time white noises and for every $(t,x) \in (0,\infty)\times \R$ and $0 \leq k \leq n$ almost surely
		\begin{equation*}
			\begin{split}
				&X_t^k(x) = \int_\R p_t(x-y) X_0^k(y)dy + \int_\R \int_0^t p_{t-s}(x-y)g^k_s(y)dsdy \\
				& \hspace{2cm}+ \int_\R \int_0^t p_{t-s}(x-y) \sqrt{X_s^k(y)}W^k(dsdy),\quad \bbP-a.s.,
			\end{split}
		\end{equation*}
		with $g^0 \equiv 0$.
	\end{definition}
	\begin{proposition}
		For any $n\geq 0$ and  $(X_0^0, \ldots X^n_0)\in \big(\mathcal{C}_{tem}^+\big)^{n+1}$ there exists a weak solution to \eqref{eq:SPDE_n} in $(\mcC(\R_+, \mcC_{tem}^+))^{n+1}$ such that, conditionally on $\mathcal{F}_\infty^{k-1}$,  $X^k$ is a solution to \eqref{eq:SPDE_k_gen} for $k \in \{1, \ldots n\}$.
	\end{proposition}
	\begin{proof}
		Equation \eqref{eq:SPDE_0_gen} is well-known to have a weak solution in $\mcC(\R_+, \mcC_{tem}^+)$ (see e.g. \cite{Shiga94}). Provided that we have already constructed a solution to \eqref{eq:SPDE_k_gen} for $k \geq 1$, the existence of solutions \eqref{eq:SPDE_k_gen} for $k+1$ in $\mcC(\R_+ , \mcC_{tem}^+)$ conditional on $\mathcal{F}_\infty^{k}$ follows from a similar argument as in \cite{Shiga94}, since $g^{k+1}$ is adapted to the filtration $(\mathcal{F}^k_t)_{t\geq 0}$ and is nonnegative and bounded and $\dot{W}^{k+1}$ is independent of $\mathcal{F}^k_\infty$.
	\end{proof}
	\begin{lemma}\label{lem:duality}
		Let $n\geq 1$ and let $(X^0, \ldots X^n)$ be a solution to \eqref{eq:SPDE_n} such that conditionally on $\mathcal{F}_\infty^{k-1}$,  $X^k$ is a solution to \eqref{eq:SPDE_k_gen} for $1\leq k \leq n$. Then, for any $1\leq k \leq n$ and any  $\psi^k, \phi^k\in \mcC_c^+(\R)$, it holds that 
		\begin{equation}\label{eq:duality_relation}
		\begin{split}
		&	\E\Big[\exp\Big(- \langle X^k_t, \phi^k \rangle - \int_0^t \langle X^k_s, \psi^k\rangle ds\Big)\Big| \mathcal{F}^{k-1}_\infty\Big] \\
		& \quad = \exp\Big(- \langle X_0^k, V_t(\phi^k, \psi^k)\rangle -\int_0^t\langle g_s^k, V_{t-s}(\phi^k, \psi^k) \rangle ds\Big), 
		\end{split}
		\end{equation} 
		where $V(\phi^k, \psi^k)$ is the solution to 
		\begin{equation*}
			\begin{split}
				&\partial_t V_t(x) = \frac{1}{2}\Delta V_t(x) - \frac{1}{2}(V_t(x))^2 + \psi^k(x), \quad t >0, x\in \R\\
				&V_0 \equiv \phi^k.
			\end{split}
		\end{equation*}
	\end{lemma}
	\begin{proof}
		Since $W^k$ is independent of $\mathcal{F}^{k-1}_t$ this follows along similar lines as the proof of \cite{Perkins02} by applying It\^o's formula. 
	\end{proof}
	\begin{proposition}
		For each $n\geq 0$ and any initial condition $(X^0_0, \ldots X^n_0) \in \big(\mathcal{C}_{tem}^+\big)^{n+1}$, the solution to \eqref{eq:SPDE_n} such that conditionally on $\mathcal{F}_\infty^{k-1}$, $X^k$ solves \eqref{eq:SPDE_k_gen} for $1\leq k \leq n$, is weakly unique in $\mcC(\R_+, \mcC_{tem}^+)^{n+1}$.
		Moreover, conditionally on $\mathcal{F}^k_\infty$ the law of $X^{k+1}$ is the law of super-Brownian motion with immigration $g^{k+1}$ for $k \in \{0, \ldots, n-1 \}$.
	\end{proposition}
	\begin{proof}
		We prove uniqueness in law via induction.
		First, note that weak uniqueness holds for  $n= 0$, since super-Brownian motion is weakly unique. Let $n \geq 1$ be arbitrary and assume that the claim is true for $n-1$ . 
		Let $(X^0, \ldots X^n)$ be any solution to \eqref{eq:SPDE_n} such that conditionally on $\mathcal{F}_\infty^{k-1}$, $X^k$ solves \eqref{eq:SPDE_k_gen} for $1\leq k \leq n$. Then, by induction, $(X^0,\ldots X^{n-1})$ has a unique law. By Lemma~\ref{lem:duality} and standard martingale problem arguments, we get that the law of $X^n$ conditioned on $\mathcal{F}_{\infty}^{n-1}$ is the unique law of super-Brownian motion with immigration $g^n$. This gives us the uniqueness in law of $(X^0,\ldots X^n)$ and the claim follows.
	\end{proof}

	\section{Cozero Set Estimates}\label{sec:cozero}
	Fix arbitrary $n\geq 1$ and let $(X^0,\ldots X^n)$ be the unique weak solution to \eqref{eq:SPDE_n} such that, conditionally on $\mathcal{F}_\infty^{k-1}$ the law of $X^k$ is the law of super-Brownian motion with immigration $g^k$ with $X_0^k \equiv 0$ for $k \in \{1,\ldots ,n\}$ and $X_0^0 = f\in \mcC_c^+(\R)$.
	An important step in our analysis is to obtain estimates for the Lebesgue measure of the cozero set of $X^k$ for $k\geq 1$.
	Namely, we find estimates for 
	\begin{equation}\label{eq:def_Jnt}
		J_t^n = \E\Big[\int_0^t\Leb(\{x \colon X_s^n(x) >0\}) ds\Big], \quad  t\geq 0.
	\end{equation}
	To this end, we employ the duality relation  to get 
	\begin{equation*}
		\begin{split}
			&\E_f\Big[\int_0^t\Leb(\{x \colon X_s^n(x) >0\}) ds\Big] = \int_0^t \int_\R\E_f\big[\mathbbm{1}_{X^n_s(x)>0} \big]dxds\\
			& = \lim_{m \rightarrow \infty} \int_0^t \int_\R\E_f\big[1-\exp(-m X^n_s(x)) \big]dxds\\
			& =  \int_0^t \int_\R\E_f\big[1-\exp(-\int_0^s\langle g^n_{r}, \mcW_{s-r}^x \rangle dr)\big]dxds\\
			& \leq \E_f\Big[ \int_0^t \int_\R\int_0^s\langle g^n_{r}, \mcW_{s-r}^x \rangle drdxds\Big],\quad t \geq 0, 
		\end{split}
	\end{equation*}
	where $\mathcal{W}^x$ is the function given by \eqref{eq:def_W}. 
	Furthermore, note that 
	\begin{equation}\label{eq:Jt0_est}
		 J_t^0 \leq 2\E_f[R^0_t]t, \quad t\geq 0,
	\end{equation}
	where
	\begin{equation*}
		R^0_t  = \sup\{|x| \colon \text{ there exists }  s\in [0,t] \text{ such that } X^0_s(x) >0 \}, \quad t \geq 0,
	\end{equation*}
    and 
    \begin{equation*}
        \E_f[R_t^0] <\infty, 
    \end{equation*}
    by \cite[Theorem 1]{Iscoe89}.
	\begin{lemma}
		For all $y\in \R$, 
		\begin{equation}\label{eq:est_W_time_int}
			\int_0^t \int_\R \mcW^x_s(y)dxds \leq \constFromEq{c}{eq:est_W_time_int}\sqrt{t}, \quad \forall t >0.
		\end{equation}
	\end{lemma}
	\begin{proof}
		By Proposition~\ref{prp:very_sing_asymp} it follows that 
		\begin{equation*}
			\mcW^x_t(y) = \frac{1}{t}\xi\Big(\frac{|x-y|}{t^{1/2}}\Big), \quad t >0, 
		\end{equation*} 
		where $\xi$ is defined as in \eqref{eq:def_W}. Thus, using \eqref{eq:est_W_int} yields
		\begin{equation*}
			\int_0^t \int_\R \mcW^x_s(y)dxds = 	\int_0^t \int_\R \mcW^y_s(x)dxds\leq  \constFromEq{c}{eq:est_W_time_int}\sqrt{t}, \quad t >0,
		\end{equation*}
		for some constant $\constFromEq{c}{eq:est_W_time_int}>0$.
	\end{proof}
	Using the above lemma, we can derive the following estimate.
	\begin{lemma}
		Let $n\ge 1$. Then it follows that 
		\begin{equation*}
			\E_f\Big[ \int_0^t \int_\R\int_0^s\langle g^n_{r}, \mcW_{s-r}^x \rangle drdxds\Big] \leq \frac{\constFromEq{c}{eq:est_W_time_int}}{2}b_1\int_0^t \frac{1}{\sqrt{t-s}} J^{n-1}_s ds,
		\end{equation*}
	for all $t>0$.
	\end{lemma}
	\begin{proof}
		Using the definition of $g^n$ yields
		\begin{equation*}
			\begin{split}
				&\E_f\Big[ \int_0^t \int_\R\int_0^s\langle g^n_{r}, \mcW_{s-r}^x \rangle drdxds\Big]\\
				& =\E_f \Big[\int_0^t \langle g^n_r, \int_r^t\langle \mcW^x_{s-r} ,1 \rangle ds \rangle dr\Big]\\
				& \leq \constFromEq{c}{eq:est_W_time_int}\E_f\Big[ \int_0^t \sqrt{t-r} \langle g^n_r, 1 \rangle dr\Big]\\
				& = \constFromEq{c}{eq:est_W_time_int} \int_0^t \sqrt{t-r} \E_f\big[\langle g^n_r, 1 \rangle \big]dr\\
				& \leq  \constFromEq{c}{eq:est_W_time_int}b_1\int_0^t \sqrt{t-r} \E_f\big[\Leb(\{ x\colon X^{n-1}_r(x)>0\})\big]dr,
			\end{split}
		\end{equation*}
		for all $t>0$.
		Applying integration by parts yields
		\begin{equation*}
			\E_f\Big[ \int_0^t \int_\R\int_0^s\langle g^n_{r}, \mcW_{s-r}^x \rangle drdxds\Big]\leq \frac{\constFromEq{c}{eq:est_W_time_int}}{2}b_1 \int_0^t \frac{1}{\sqrt{t-s}} \int_0^s  \E_f\big[\Leb(\{x \colon X^{n-1}_r(x)>0\})\big]drds, \quad t >0, 
		\end{equation*}
		which finishes the proof. 
	\end{proof}
	
	The above considerations lead to the bound 
	\begin{equation}\label{eq:iterative_J_t_est}
		\begin{split}
				J_t^n 
            &\leq 2\Big(\frac{b_1\constFromEq{c}{eq:est_W_time_int}}{2}\Big)^n\Big(\int_0^t\frac{1}{\sqrt{t-t_{n-1}}}\int_0^{t_{n-1}} \frac{1}{\sqrt{t_{n-1}-t_{n-2}}} \\
            & \qquad \qquad \qquad \cdots \int_0^{t_1} \sqrt{t_1-t_0}\E_f[\Leb(\{x \colon X^0_{t_0}(x)>0\})]dt_0 \cdots dt_{n-1}\Big)\\
			& \leq 4\E_f[R^0_t]\Big(\frac{b_1\constFromEq{c}{eq:est_W_time_int}}{2}\Big)^n\int_0^t\frac{1}{\sqrt{t-t_{n-1}}}\int_0^{t_{n-1}} \frac{1}{\sqrt{t_{n-1}-t_{n-2}}} \cdots \int_0^{t_1} \sqrt{t_1-t_0}dt_0 \cdots dt_{n-1},
		\end{split}
	\end{equation}
	for $t>0$ and $n\geq 2$.
	\begin{lemma}\label{lem:K_t_est}
		Let $n\geq 2$ and $t \geq 0$. 
		Then we have the estimate 
		\begin{equation*}
        \begin{split}
			K_t^n &= 	\int_0^t\frac{1}{\sqrt{t-t_{n-1}}}\int_0^{t_{n-1}} \frac{1}{\sqrt{t_{n-1}-t_{n-2}}} \cdots \int_0^{t_1} \sqrt{t_1-t_0}dt_0 \cdots dt_{n-1} \\
            &=  \frac{2}{3}2^{n-1}\Big(\prod_{k=2}^{n}w_{k+2}\Big)t^{(n+2)/2},
        \end{split}
		\end{equation*}
		where $w_i = \int_0^{\pi/2} \sin^i(r)dr$ for $i \geq 2$.
	\end{lemma}
	\begin{proof}
		For $n = 2$ we arrive at 
		\begin{equation*}
			K^2_t = \frac{2}{3} \int_0^t \frac{t_1^{3/2}}{\sqrt{t-t_1}}dt_1 = \frac{2}{3}2 t^{4/2}w_{4},
		\end{equation*}
		by using that 
		\begin{equation}\label{eq:integral_est}
			\int_0^t \frac{r^{k/2}}{\sqrt{t-r}}dr = 2 t^{(k+1)/2}\int_0^{\pi/2}\sin^{k+1}(r)dr = 2 t^{(k+1)/2}w_{k+1},
		\end{equation}
		for $k \geq 2$. 
		For $n\geq 3$ we iteratively get 
		\begin{equation*}
        \begin{split}
			K_t^n &= \int_0^t \frac{1}{\sqrt{t-t_{n-1}}} K^{n-1}_{t_{n-1}} dt_{n-1} = \frac{2}{3}2^{n-2}\Big(\prod_{k=2}^{n-1}w_{k+2}\Big)\int_0^t\frac{t_{n-1}^{(n+1)/2}}{\sqrt{t-t_{n-1}}}dt_{n-1}\\
            &= \frac{2}{3}2^{n-1}\Big(\prod_{k=2}^{n}w_{k+2}\Big)t^{(n+2)/2}.
        \end{split}
		\end{equation*}
	\end{proof}
	
	By applying Lemma~\ref{lem:K_t_est} to \eqref{eq:iterative_J_t_est} we obtain
	\begin{equation}\label{eq:est_Jnt}
		J_t^n = \E_f\Big[\int_0^t\Leb(\{x \colon X_s^n(x) >0\}) ds\Big] \leq\frac{4}{3}\E_f[R^0_t] \constFromEq{c}{eq:est_W_time_int}^nb_1^n \Big( \prod_{k = 2}^n w_{k+2} \Big)t^\frac{n+2}{2}
	\end{equation}
	for $n\geq 2$ and similarly, we obtain
    \begin{equation}\label{eq:est_J1t}
        J_t^1 \leq \frac{4}{3} \E_f[R_t^0] \constFromEq{c}{eq:est_W_time_int} b_1 t^\frac{3}{2}.
    \end{equation}

	\section{Bounds on the Radius of the Support}\label{sec:radius}
	In order to prove the compact support property we will prove bounds for the expectation of 
	\begin{equation*}
		R^n_T = \sup\{|x| \colon \text{ there exists }  t\in [0,T] \text{ such that } \bar{X}_t^n(x) >0 \}, 
	\end{equation*}
	which are uniform in $n$ and depend only on $b_1$, $T$ and $\E_f[R_T^0]$. This is the content of the following lemma. 
	\begin{lemma}\label{lem:sup_supports_exp}
		For any $T>0$ we have 
		\begin{equation}\label{eq:support_estimate_1}
			\sup_{n\geq 1} \E_f[R^n_T] \leq \E_f[R^0_T] +T^{1/2} \sqrt{2\pi b_1 \constFromEq{c}{eq:est_v_infty_1}}\sqrt{\E_f[R^0_T]}+ \sqrt{\E_f[R_T^0]} \sum_{i=1}^\infty a_i(T) <\infty, 
		\end{equation}
		where 
		\begin{equation}\label{eq:def_ai}
			a_i(T):=   \sqrt{\frac{4\pi \constFromEq{c}{eq:est_v_infty_1}}{3} \constFromEq{c}{eq:est_W_time_int}^ib_1^{i+1} \Big( \prod_{k = 2}^i w_{k+2} \Big)T^\frac{i+2}{2}},
		\end{equation}
    with the convention that $\prod_{k=2}^1w_{k+2} =1$.
	\end{lemma}
	\begin{proof}
    First note that $\E[R_t^0] <\infty$ by \cite[Theorem 1]{Iscoe89}.
	By definition, we have the decomposition
	\begin{equation*}
		\E_f[R^{n}_T ] = \E_f[R_T^{n-1}] + \E_f[R_T^{n}-R^{n-1}_T],
	\end{equation*}
	where both $R_T^{n-1} \geq 0$ and $R_T^{n}-R^{n-1}_T\geq 0$.
	This means that our main task lies in finding bounds for 
	\begin{equation*}
		\E_f[R_T^{n}-R^{n-1}_T] = \int_{0}^\infty \bbP_f(R_T^n > r+R^{n-1}_T)dr,
	\end{equation*}
	where the above identity follows from the fact that $R^n_T \geq R^{n-1}_T$ almost surely.
	Noting that $R^n_T >R^{n-1}_T+r$ if and only if $\int_0^T\langle X^n_t, \mathbbm{1}_{[-R^{n-1}_T-r, R^{n-1}_T+r]^c}\rangle dt >0$ it is enough to derive bounds for
	\begin{equation}\label{eq:charge_Rn_est}
		\begin{split}
			&\bbP_f\Big(\int_0^T\langle X^n_t, \mathbbm{1}_{[-R^{n-1}_T-r, R^{n-1}_T+r]^c}\rangle dt >0\Big)\\
			& \qquad  = \lim_{m \rightarrow \infty} \E _f\Big[ 1-\exp\Big(-m\int_0^T\langle X^n_t, \mathbbm{1}_{[-R^{n-1}_T-r, R^{n-1}_T+r]^c}\rangle dt \Big)\Big].
		\end{split}
	\end{equation}
	Now, by conditioning on $\mathcal{F}_\infty^{n-1}$, approximation and using the duality formula \eqref{eq:duality_relation} we  arrive at 
	\begin{equation*}
		\lim_{m\rightarrow \infty}\E_f\Big[	\exp\Big(-m\int_0^T\langle X^n_t, \mathbbm{1}_{[-R^{n-1}_T-r, R^{n-1}_T+r]^c}\rangle dt \Big)\Big] = \E_f\Big[\exp\Big(- \int_0^T \langle g^n_{s}, V^{r+R^{n-1}_T}_{T-s}\rangle ds\Big)\Big].
	\end{equation*}
	Here, $Vr$ for $r >0$ is defined as 
	\begin{equation}\label{eq:limit_vkm}
		V^r_t(x) =\lim_{m \rightarrow \infty}V_t^{m,r}(x),
	\end{equation}
	where $V^{m,r}$ is the unique  mild solution to \eqref{eq:parabolic_m}.  
	By definition, we know that the support of $g^n_s$ is contained in $[-R^{n-1}_s, R^{n-1}_s]$ for all $s\in [0,T]$. Together with \eqref{eq:est_v_infty_1} this allows us to get the bound
	\begin{equation*}
		\int_0^T\langle g^n_{s}, V^{r+R^{n-1}_T}_{T-s}\rangle ds \leq \constFromEq{c}{eq:est_v_infty_1}r^{-2}\int_0^Tb_1\Leb(\{ x\colon X^{n-1}_s(x)>0\})ds.
	\end{equation*} 
	Using this, \eqref{eq:charge_Rn_est} and \eqref{eq:est_Jnt}, we get for $n \geq 2$ that  
	\begin{equation*}
		\begin{split}
			&\int_0^\infty\bbP(R_T^n > r+R^{n-1}_T)dr \leq \E_f \Big[\int_0^\infty
			1-\exp\Big(-r^{-2}\constFromEq{c}{eq:est_v_infty_1}b_1\int_0^T \Leb(\{x \colon X_s^{n-1}(x) >0\}ds\Big)dr\Big]\\
            &\leq \sqrt{\pi b_1\constFromEq{c}{eq:est_v_infty_1}} \sqrt{J_T^{n-1}},\\
		\end{split}
	\end{equation*}
	and similarly for $n = 1$ that 
	\begin{equation*}
		\begin{split}
			&\int_0^\infty\bbP(R_T^1 > r+R^{0}_T)dr \leq \E_f \Big[\int_0^\infty
			1-\exp(-\constFromEq{c}{eq:est_v_infty_1}b_1r^{-2}\int_0^T\Leb(\{x \colon X^{0}_s(x)>0\}) ds)dr\Big]\\
			&\leq \sqrt{\pi b_1\constFromEq{c}{eq:est_v_infty_1}} \sqrt{J_T^0}\\
			&  \leq T^{1/2} \sqrt{2\pi b_1 \constFromEq{c}{eq:est_v_infty_1}}\sqrt{\E[R^0_T]}.
		\end{split}
	\end{equation*}
	Overall, we obtain
	\begin{equation}\label{eq:support_est_prelim}
	\begin{split}
		\E_f[R^n_T] &\leq \E_f[R_T^0] +\E_f[R^1_T-R^0_T]+ \sqrt{\pi b_1\constFromEq{c}{eq:est_v_infty_1}}\sum_{i = 2}^\infty\sqrt{J^{i-1}_T} \\
	&\leq  \E_f[R_T^0] +  T^{1/2} \sqrt{2\pi b_1 \constFromEq{c}{eq:est_v_infty_1}}\sqrt{\E_f[R^0_T]} + \sqrt{\pi b_1\constFromEq{c}{eq:est_v_infty_1}}\sum_{i = 2}^\infty\sqrt{J^{i-1}_T} ,
	\end{split}
	\end{equation}
	and by using \eqref{eq:est_Jnt} and \eqref{eq:est_J1t} we get 
	\begin{equation}\label{eq:support_estimate_2}
		\sum_{i = 2}^\infty\sqrt{J^{i-1}_T} \leq \sqrt{\E_f[R^0_T]} \sum_{i =2}^\infty\sqrt{\frac{4}{3}\constFromEq{c}{eq:est_W_time_int}^{i-1}b_1^{i-1} \Big( \prod_{k = 2}^{i-1} w_{k+2} \Big)T^\frac{i+1}{2}}
		 = \sqrt{\E_f[R_T^0]}\sum_{i=1}^\infty d_i(T),
	\end{equation}
	where 
	\begin{equation}\label{eq:def_di}
		d_i(T) =\sqrt{\frac{4}{3}\constFromEq{c}{eq:est_W_time_int}^ib_1^i \Big( \prod_{k = 2}^i w_{k+2} \Big)T^\frac{i+2}{2}}
	\end{equation}
    and we use the convention $\prod_{k=2}^1w_{k+2} =1$.
	By substitution, we get that 
	\begin{equation*}
		w_i= \frac{\Gamma(\frac{i+1}{2})\Gamma(\frac{1}{2})}{2\Gamma(\frac{i}{2}+1)},
	\end{equation*}
	for $i\geq 2$, where $\Gamma$ is the gamma function. Using Stirling's approximation for the gamma function thus reveals that $w_i \rightarrow 0$ as $i \rightarrow \infty$.
	We thus get
	\begin{equation*}
		\frac{d_{i+1}}{d_i} = \sqrt{\constFromEq{c}{eq:est_W_time_int}b_1T^{1/2}}\sqrt{w_{i+3}} \rightarrow 0,
	\end{equation*}
	as $i\rightarrow \infty$, which shows that the series on the right-hand side of \eqref{eq:support_estimate_1} is convergent. Thus, \eqref{eq:support_est_prelim}, \eqref{eq:support_estimate_2} and \eqref{eq:def_di} imply \eqref{eq:support_estimate_1} and the proof is finished. 
	\end{proof}
	
	\section{Proof of Theorem~\ref{th:convergence_X_bar_n}}\label{sec:series_rep}

	In this section, we prove Theorem~\ref{th:convergence_X_bar_n}. To this end we derive the following preliminary results. 
	\begin{proposition}\label{prop:l1_conv}
		Let $\phi\in \mcC_b^2(\R)$ and $T>0$. Then it follows that the sequence $\langle \bar{X}^n,\phi \rangle$ converges in $L^1(\Omega, \mcC([0,T]))$. 
	\end{proposition}
	
	\begin{proof}
		For any $n \geq \ell \geq 1$, let us denote 
		\begin{equation*}
			\bar{X}^{n,\ell} = \sum_{k=\ell}^n X^k, 
		\end{equation*}
		where $(X^0, \ldots, X^n)$ is the solution to \eqref{eq:SPDE_n} with $X_0^k \equiv 0$ for $k \geq 1$ and $X^0_0 \equiv f$.  It is sufficient to prove that 
		\begin{equation}\label{eq:cauchy_1}
			\lim_{n,\ell \rightarrow \infty, n\geq \ell}\E_f\big[\sup_{0 \leq t \leq T}	\langle\bar{X}_t^{n,\ell}, \phi \rangle \big] = 0,
		\end{equation}
        for nonnegative $\phi \in \mcC_b^2(\R)$.
		After extending the probability space and using the independence of white noises $\dot{W}^0, \ldots, \dot{W}^n$, it is easy to see  that $\bar{X}^{n,\ell}$ is a solution of 
		\begin{equation*}
			d_t\bar{X}_t^{n,\ell}(x) = \frac{1}{2}\Delta \bar{X}_t^{n,\ell}(x) +\sum_{k=\ell}^ng^k(t,x)+ \sqrt{\bar{X}_t^{n,\ell}(x)}\dot{W}(t,x), \quad \bar{X}_0^{n,\ell} = 0, x\in \R, t\geq 0,
		\end{equation*}
		where $\dot{W}$ is space-time white noise and $g^k$ is defined as in \eqref{eq:def_g_k}. Thus, $\bar{X}^{n,\ell}$ in particular satisfies
		\begin{equation*}
			\langle \bar{X}_t^{n,\ell}, \phi \rangle = \int_0^t \frac{1}{2}\langle \bar{X}_s^{n,\ell}, \Delta \phi \rangle ds + \int_0^t \langle \sum_{k = \ell}^n g^k_s, \phi \rangle ds + M_t^{n,\ell}(\phi), 
		\end{equation*}
		where $M_t^{n,\ell}(\phi)$ is a square integrable martingale with quadratic variation given by 
		\begin{equation*}
			\langle M^{n,\ell}(\phi) \rangle_t = \int_0^t \langle \bar{X}^{n,\ell}_s, \phi^2\rangle ds, \quad t \geq 0.
		\end{equation*}
		Thus, we arrive at the bound
		\begin{equation*}
			\begin{split}
				&\E_f\big[\sup_{0 \leq t \leq T}	\langle\bar{X}_t^{n,\ell}, \phi \rangle \big] \\
				& \leq \E_f\Big[\int_0^T \frac{1}{2}|\langle \bar{X}_s^{n,\ell}, \Delta \phi \rangle| ds \Big] + \E_f\Big[  \int_0^T \langle \sum_{k = \ell}^n g^k_s, \phi \rangle ds \Big] + \E[\sup_{0 \leq t \leq T} |M_t^{n,\ell}(\phi)|]\\
				& =: I^1_{n,\ell, T} + I^2_{n,\ell, T} + I^3_{n,\ell,T}. 
			\end{split}
		\end{equation*}
		The first term can be estimated by 
		\begin{equation*}
			I^1_{n,\ell, T}  \leq \frac{1}{2}\| \Delta \phi\|_\infty \E\Big[\int_0^T \langle \bar{X}_s^{n,\ell}, 1 \rangle ds \Big] \leq  T\frac{1}{2}\| \Delta \phi\|_\infty \E\Big[\int_0^T \langle \sum_{k= \ell}^ng^k_s,1 \rangle ds \Big] \leq \frac{b_1T}{2}\| \Delta \phi\|_\infty\sum_{k = \ell}^n J^{k-1}_T,
		\end{equation*}
		where $J^{k}_T$ for $k\geq 0$ is defined as in \eqref{eq:def_Jnt}. By an argument similar to that after  \eqref{eq:support_estimate_2} we get that
		\begin{equation*}
			\sum_{k = 1}^\infty J_T^k <\infty.
		\end{equation*}
		From this, we can deduce that 
		\begin{equation*}
			I^1_{n,\ell, T} \rightarrow 0,
		\end{equation*}
		as $n,\ell\rightarrow \infty$. Similarly, we get that 
		\begin{equation*}
			I^2_{n,\ell, T} \rightarrow 0, 
		\end{equation*}
		as $n,\ell\rightarrow \infty$. Thus, it only remains to prove the convergence of $I^3_{n,\ell, T}$.  The Burkholder-Davis-Gundy inequality yields 
		\begin{equation*}
			I^3_{n,\ell,T} \leq C \E_f\Big[ \int_0^T \langle \bar{X}^{n,\ell }_s, \phi^2\rangle ds\Big]^{1/2} \leq C \|\phi\|_\infty \E_f\Big[ \int_0^T \langle \bar{X}^{n,\ell }_s, 1\rangle ds\Big]^{1/2},
		\end{equation*}
		which, once again using that 
		\begin{equation*}
			\E_f \Big[ \int_0^T \langle \bar{X}^{n,\ell}_s,1 \rangle ds\Big] \leq T\E_f\Big[ \int_0^T \sum_{k =\ell}^n \langle g^k_s,1\rangle ds\Big]
		\end{equation*}
		demonstrates, in the same way as above, that 
		\begin{equation*}
			I^3_{n,\ell, T} \rightarrow 0, 
		\end{equation*}
		as $n,\ell\rightarrow \infty$.
		This shows \eqref{eq:cauchy_1} and finishes the proof. 
	\end{proof}
	

	\begin{lemma}\label{lem:as_conv}
		There exists a jointly measurable map $\bar{X}\colon \Omega \times \R_+\times \R \rightarrow \R_+\cup \{\infty\}$  such that $\bbP$-a.s. we have 
		\begin{equation*}
			\lim_{n \rightarrow  \infty} \bar{X}^n_t(x) = \bar{X}_t(x), \quad \forall (t,x) \in \R_+\times \R.
		\end{equation*} 
		Furthermore, $\bbP$-a.s.
		\begin{equation*}
		\sum_{k = 1}^n g^k(t,x) \rightarrow b_1 \mathbbm{1}_{\bar{X}_t(x)>0}, \quad (t,x) \in \R_+\times\R,
		\end{equation*}
		as $n \rightarrow \infty$. 
	\end{lemma}
	\begin{proof}
		The  almost sure convergence of $\bar{X}^n$ follows from the fact that $\bar{X}^n$ takes values in $\mcC(\R_+, \mcC_{tem}^+)$ and 
		\begin{equation*}
			\bar{X}^n_t(x) \leq \bar{X}^{n+1}_t(x), 
		\end{equation*}
		almost surely, for all $n\geq 0$ and $(t,x) \in \R_+\times\R$. Indeed, it easily follows that we find a measurable set $A\subset \Omega$ such that $\bbP(A) = 1$ and   
		\begin{equation*}
			\bar{X}^n_t(\omega, x) \leq \bar{X}^{n+1}_t(\omega, x), \quad \omega \in A, (t,x) \in \R_+\times\R, 
		\end{equation*}
		and thus the convergence follows by monotonicity.
		It remains to prove that
		\begin{equation*}
			g(t,x):=\sum_{k = 1}^\infty g^k(t,x) = b_1\mathbbm{1}_{\bar{X}_t(x)>0}, \quad (t,x)\in\R_+\times \R
		\end{equation*} 
		on $A$. Let $\omega\in A$ and $(t,x) \in \R_+\times \R$ be arbitrary.
		Note that $g$ can only take the values $0$ or $b_1$. Suppose that $g(t,x) =b_1$. This implies that there exists $n^\ast\in \N$ such that $\bar{X}_t^{n^\ast}(x)>0$. Thus, by monotonicity, this also implies that $\bar{X}_t(x) >0$. On the other hand, suppose that $g(t,x) = 0$. This implies that $X^k_t(x) = 0$ for all $k \geq 0$ and thus $\bar{X}_t(x) = 0$. This completes the proof.
		
	\end{proof}
    \begin{proof}[Proof of Theorem~\ref{th:convergence_X_bar_n}]
		As a first step in the proof we show that Proposition~\ref{prop:l1_conv} implies that the sequence $\bar{X}^n$ converges in $L^1(\Omega, \mcC([0,T], (\mathcal{M}_F^+, d_{bl})))$ to a limit $\tilde{X}$.   Since the space $L^1(\Omega, \mcC([0,T], (\mathcal{M}_F^+, d_{bl})))$ is complete, it is sufficient to prove that $\bar{X}^n$ is a Cauchy sequence. That is, we need to prove
		\begin{equation}\label{eq:cauchy_seq}
			\begin{split}
			&\lim_{n , m \rightarrow \infty, n\geq m}	\E[\sup_{0 \leq t \leq T} d_{bl}(\bar{X}^n_t, \bar{X}^m_t)]\\
			& \quad = \lim_{n , m \rightarrow \infty, n\geq m}	\E[\sup_{0 \leq t \leq T} \sup_{\phi\in \mathcal{L}}\langle \bar{X}_t^n-\bar{X}_t^m, \phi \rangle ] = 0,
			\end{split}
		\end{equation}
		where 
		\begin{equation*}
			\mathcal{L} = \{ \phi \colon \R \rightarrow \R | \|\phi\|_\infty \leq 1, \|\phi\|_{Lip} \leq 1\}.
		\end{equation*}
        Since $\bar{X}^n_t-\bar{X}^m_t \geq 0$ it follows that 
        \begin{equation*}
            \lim_{n , m \rightarrow \infty, n\geq m}	\E[\sup_{0 \leq t \leq T} d_{bl}(\bar{X}^n_t, \bar{X}^m_t)] = \lim_{n , m \rightarrow \infty, n\geq m}	\E[\sup_{0 \leq t \leq T} \langle \bar{X}_t^n-\bar{X}_t^m, 1\rangle ],
        \end{equation*}
        and thus \eqref{eq:cauchy_seq} follows by Proposition~\ref{prop:l1_conv}.
		  \\$~$\\
		In a second step, we verify that for each $\phi\in \mathcal{C}_c^2$ the limit $\tilde{X}$ of $(\bar{X}^n)_{n\geq 1}$  in $L^1(\Omega, \mcC([0,T], \mathcal{M}_F^+))$ satisfies 
		\begin{equation}\label{eq:mp_lim}
			\langle \tilde{X}_t, \phi \rangle =\langle X_0, \phi \rangle + \int_0^t \frac{1}{2}\langle \tilde{X}_s, \Delta \phi \rangle ds + \int_0^t \langle \sum_{k = 1}^\infty g^k_s, \phi \rangle ds + M_t(\phi), \quad t \in [0,T],
		\end{equation}
		where $M(\phi)$ is a continuous martingale with quadratic variation given by 
		\begin{equation}\label{eq:quad_var_lim}
				\langle M(\phi) \rangle_t = \int_0^t \langle \tilde{X}_s , \phi^2\rangle ds, \quad t\in [0,T].
		\end{equation}
		To this end, note that $\bar{X}^n$ satisfies 
		\begin{equation}\label{eq:martingale_for_sum}
			\langle \bar{X}^n_t, \phi \rangle =\langle X_0, \phi \rangle + \int_0^t \frac{1}{2}\langle \bar{X}^n_s, \Delta \phi \rangle ds + \int_0^t \langle \sum_{k = 1}^n g^k_s, \phi \rangle ds + M_t^n(\phi), \quad t \in [0,T],
		\end{equation}
		where $M^n$ is a continuous square integrable martingale with quadratic variation given by 
		\begin{equation}\label{eq:quadratic_var}
				\langle M^n(\phi) \rangle_t = \int_0^t \langle \bar{X}^n_s , \phi^2\rangle ds, \quad t\in [0,T].
		\end{equation}
		By Proposition~\ref{prop:l1_conv} we get that $	\langle \bar{X}^n_t, \phi \rangle$ and $\int_0^t \frac{1}{2}\langle \bar{X}^n_s, \Delta \phi \rangle ds $ converge in $L^1(\Omega, \mcC([0,T]))$ to  $	\langle \tilde{X}_t, \phi \rangle$ and $\int_0^t \frac{1}{2}\langle \tilde{X}_s, \Delta \phi \rangle ds $, respectively. Furthermore, by the dominated convergence theorem we get that 
		$\int_0^t\sum_{k = 1}^n \langle g^k_s, \phi \rangle ds$ converges in $L^1(\Omega, \mcC([0,T]))$ to  $\int_0^t\sum_{k = 1}^\infty \langle g^k_s, \phi \rangle ds$.  Note that  $M^n(\phi)$ is a martingale with respect to $(\mathcal{F}^\infty_t)_{t\geq 0}$, where $\mathcal{F}^\infty_t = \sigma(\cup_{k = 0}^\infty\mathcal{F}_t^k)$. This implies that $M^n(\phi)$ converges in $L^1(\Omega, \mcC([0,T]))$ to a continuous martingale $M(\phi)$.
        Furthermore, by the Burkholder--Davis--Gundy inequality and the orthogonality of the martingales corresponding to the independent summands \(X^k\), for \(m<n\), we get
\[
\mathbf E\Big[\sup_{t\le T}|M_t^n(\varphi)-M_t^m(\varphi)|^2\Big]
\le C_\varphi\,\mathbf E\int_0^T\Big\langle \sum_{k=m+1}^n X_s^k,1\Big\rangle ds \rightarrow 0,
\]
where the convergence follows from Proposition~\ref{prop:l1_conv}; hence \(M^n(\varphi)\to M(\varphi)\) in \(L^2(\Omega,C([0,T]))\).  Applying Proposition~\ref{prop:l1_conv} again to \(\varphi^2\), we have
\[
\int_0^t\langle \bar X_s^n,\varphi^2\rangle ds
\rightarrow
\int_0^t\langle \widetilde X_s,\varphi^2\rangle ds
\quad\text{in }L^1(\Omega,C([0,T])),
\]
which implies 
\[
\langle M(\varphi)\rangle_t
=
\int_0^t\langle \widetilde X_s,\varphi^2\rangle ds .
\]
        This demonstrates that $\tilde{X}$ satisfies the martingale problem \eqref{eq:mp_lim}. \\
		In the last step we deduce that the limit $\tilde{X}$ indeed solves \eqref{eq:SPDE1_sup}.  Denote by $\bar{X}$ the process from Lemma~\ref{lem:as_conv}. Clearly, we get 
		\begin{equation*}
			\tilde{X}_t(dx) = \bar{X}_t(x)dx, \quad \forall t \in [0,T], 
		\end{equation*}
		almost surely. Thus, by using Lemma~\ref{lem:as_conv}, we obtain that 
		\begin{equation*}
			\langle \tilde{X}_t, \phi \rangle =\langle X_0, \phi \rangle + \int_0^t \frac{1}{2}\langle \tilde{X}_s, \Delta \phi \rangle ds + \int_0^t \langle b_1 \mathbbm{1}_{\tilde{X}_s(\cdot)>0}, \phi \rangle ds + M_t(\phi), \quad t \in [0,T].
		\end{equation*}
		Furthermore, along similar lines as \cite{Shiga94} it follows that $(\bar{X}^n)_{n\geq 1}$ is tight in $\mcC([0,T], \mcC_{tem}^+)$. This implies that $\tilde{X} \in \mcC([0,T], \mcC_{tem}^+)$ almost surely.  
		Finally, using standard arguments as in \cite[Proof of Lemma 2.4]{Konno88} and after extending the filtered probability space so that it supports an additional, independent space-time white noise, it is possible to show that $\tilde{X}$ satisfies \eqref{eq:SPDE1_sup} and thus $\tilde{X}$ is the unique weak solution to \eqref{eq:SPDE1_sup} in $\mcC([0,T], \mcC_{tem}^+)$.
	\end{proof}
	
	\section{Proof of Proposition~\ref{prop:expectation_RT} and Theorem~\ref{th:compact_support}}\label{sec:compact_support}
	In this section, we prove Proposition~\ref{prop:expectation_RT} and Theorem~\ref{th:compact_support}.
	\begin{proof}[Proof of Proposition~\ref{prop:expectation_RT}]
        If $f \equiv 0$, then the identically zero process is a weak solution of \eqref{eq:SPDE1_sup}, and weak uniqueness implies that $X \equiv 0$ almost surely. Thus, we can assume without loss of generality that $f \neq 0$.
        
		By Theorem~\ref{th:convergence_X_bar_n}, the unique weak solution $X$ to \eqref{eq:SPDE1_sup} can be constructed as a monotone limit of $(\bar{X}^n)_{n\geq 1}$ and the maximal radius of the support of $X$ is a monotone limit of the maximal radius of the supports of $\bar{X}^n$.
		Thus, by the monotone convergence theorem and Lemma~\ref{lem:sup_supports_exp} we find that 
		\begin{equation}\label{eq:RT_exp}
			\E_f[R_T] = \lim_{n \rightarrow \infty}\E_f[R^n_T] \leq A(T)\Big(\E_f[R^0_T]+ \sqrt{\E_f[R^0_T]}\Big),
		\end{equation}
		where $A(T) = 1+ \sqrt{2\pi b_1 \constFromEq{c}{eq:est_v_infty_1}}T^{1/2} + \sum_{i =1}^\infty a_i(T)$ with $a_i$ given by \eqref{eq:def_ai}. 
		Thus, our main task is to derive a suitable estimate for $\E[R_T^0]$. To this end, we proceed in a similar manner as in the proof of \cite[Theorem 3.3]{Dawson89}. Note that $R_0^0= R_0$ and thus we get for $r> R_0$ that
		\begin{equation*}
			\mathbb{P}_f(R^0_T >r) \leq 1-\exp(-\langle f, V^r_T \rangle),
		\end{equation*}
		where  $V^r_T(x) = \lim_{m \rightarrow \infty} V^{m,r}_T(x)$ for $x\in \R$ and $V^{m,r}$ is the solution to \eqref{eq:parabolic_m}.
		First, we consider  $T <1$. By applying \eqref{eq:est_v_infty_1} and \eqref{eq:ext_v_infty_2}  we find that 
		\begin{equation*}
			\begin{split}
				\E_f[R_T^0] &\leq R_0+ \int_{R_0}^\infty 1-\exp(- \langle f, V^r_T \rangle) dr\\
				& \leq R_0 + \int_{R_0}^{R_0+2T^{1/4}} 1-\exp(- \langle f, V^r_T\rangle) dr + \int_{R_0+2T^{1/4}}^\infty  1-\exp(- \langle f, V^r_T \rangle) dr\\
				& \leq R_0 +\int_{R_0}^{R_0+2T^{1/4}} 1- \exp(-\constFromEq{c}{eq:est_v_infty_1}m_f (r-R_0)^{-2})dr \\
				& \qquad \qquad + \int_{R_0+2T^{1/4}}^\infty \constFromEq{c}{eq:ext_v_infty_2}m_f (r-R_0) T^{-\frac{3}{2}} \exp\Big(-\frac{(r-R_0)^2}{2T}\Big)dr\\
				&= R_0 + I_1+I_2.
			\end{split}
		\end{equation*} 
		First, we estimate $I_1$. We find that
		\begin{equation*}
			\begin{split}
				&I_1 =\sqrt{c m_f}  \int_0^{\frac{2T^{1/4}}{\sqrt{c m_f}}} 1 - \exp(- r^{-2})dr\\
				& \leq \sqrt{cm_f} \frac{2T^{1/4}}{\sqrt{c m_f}} (1- \exp(- \frac{c m_f}{2 T^{1/2}})) + \textcolor{black}{2}\sqrt{cm_f } \int_{\frac{\sqrt{cm_f}}{2T^{1/4}}}^\infty \exp(-x^2)dx\\
				& \leq T^{-1/4} cm_f + 2\sqrt{c m_f}, 
			\end{split}
		\end{equation*}
		as well as $I_1 \leq 2T^{1/4}$ and thus together
		\begin{equation*}
			I_1 \leq \min(2T^{1/4}, T^{-1/4}cm_f + \textcolor{black}{2}\sqrt{cm_f}).
		\end{equation*}
		Now, we estimate $I_2$. We get 
		\begin{equation*}
			I_2 \leq cm_f \int_{2 T^{1/4}}^\infty r T^{-3/2} \exp\Big( \frac{-r^2}{2T}\Big)dr \leq cm_f \frac{1}{\sqrt{T}}\exp(-\frac{2}{\sqrt{T}}).
		\end{equation*}
		Thus, for $T<1$ we arrive at
		\begin{equation}\label{eq:RT0_1}
		\begin{split}
			\E_f[R_T^0] &\leq R_0 + \min(2T^{1/4}, T^{-1/4}cm_f + 2\sqrt{cm_f}) + cm_f \frac{1}{\sqrt{T}} \exp(-\frac{2}{\sqrt{T}}) \\
			&\leq  R_0 + \min(2T^{1/4}, T^{-1/4}cm_f + \textcolor{black}{2}\sqrt{cm_f}) + cm_f\sqrt{T}.
		\end{split}
		\end{equation}
		When $T \geq 1$ we similarly arrive at 
		\begin{equation}\label{eq:RT0_2}
			\begin{split}
				\E_f[R_T^0] &\leq R_0 + \int_{R_0}^\infty 1-\exp(- \langle f, V^r_T \rangle) dr\\
				& \leq R_0 +\int_{R_0}^{R_0+3T^{1/2}} 1- \exp(-cm_f (r-R_0)^{-2})dr\\
                & \qquad + \int_{R_0+3T^{1/2}}^\infty cm_f (r-R_0) T^{-\frac{3}{2}} \exp\Big(-\frac{(r-R_0)^2}{2T}\Big)dr\\
				& \leq R_0+ cm_f +\sqrt{cm_f}.
			\end{split}
		\end{equation}
		Set 
		\begin{equation*}
			\beta^1(T, m_f) = \mathbbm{1}_{T<1}\big(\min(2T^{1/4}, T^{-1/4}cm_f + \textcolor{black}{2}\sqrt{cm_f})+ cm_f\sqrt{T} \big) + \mathbbm{1}_{T \geq 1}\big(cm_f +\sqrt{cm_f}\big) .
		\end{equation*}
		Then, \eqref{eq:RT0_1}, \eqref{eq:RT0_2} together with \eqref{eq:RT_exp} imply
		\begin{equation*}
			\E_f[R_T] \leq A(T)(R_0 + \sqrt{R}_0 + \beta^1(T, m_f) +  \sqrt{\beta^1(T, m_f)}). 
		\end{equation*}
		Setting 
		\begin{equation*}
			\beta(T,m_f) = \beta^1(T, m_f) +  \sqrt{\beta^1(T, m_f)}
		\end{equation*}
		yields the assertion.
	\end{proof}
	Using Proposition~\ref{prop:expectation_RT}, the proof of Theorem~\ref{th:compact_support} is almost immediate.
	\begin{proof}[Proof of Theorem~\ref{th:compact_support}]
		By Proposition~\ref{prop:expectation_RT} we find that 
		\begin{equation*}
			\bbP_f(X_t \in \mcC_c^+(\R), \forall 0 \leq t \leq T) = 1, 
		\end{equation*}
		for all $T>0$, which implies the conclusion.
	\end{proof}

    \section{Extinction -- Proof of Theorem~\ref{prop:extinction}} \label{sec:extinction}
		In this section, we prove Theorem~\ref{prop:extinction}.
	The proof relies on Proposition~\ref{prop:expectation_RT}, which allows for comparisons with squared Bessel processes. 
	\begin{lemma}\label{lem:extinction_1}
		Let $\eta \in (0,1/2)$ and $b_1>0$ and  let $c_\eta >0$ be such that 
        \begin{equation}\label{eq:mu_nuf_cond}
			\frac{1}{\Gamma(\mu)}\int_{c_\eta}^\infty u^{\mu-1}\ee^{-u}du \geq \frac{3}{4},
		\end{equation}
        for $\mu = 1-2\eta$.
        Assume that $f\in \mcC_c^+$ with $f\neq 0$, $m_f = \langle f,1\rangle$ and $K_f = \mathrm{diam}(\supp(f))/2$ is such that the following is satisfied
		\begin{enumerate}
			\item It holds that 
		\begin{equation*}
			m_f  = \langle f, 1\rangle < \frac{c_\eta}{2}.
		\end{equation*}
		\item For $\gamma(1,K_f, m_f) := A(1)(K_f+\sqrt{K_f} +\beta(1,m_f))$ with $A$ and $\beta$ from \eqref{eq:expectation_RT_est}, it holds that 
		\begin{equation}\label{eq:gamm_cond}
			\frac{b_1}{\eta}\gamma(1, K_f, m_f) \leq 10^{-2}.
		\end{equation}
		\end{enumerate}
		Then it holds for all $t\geq \frac{2m_f}{c_\eta}$ that 
		\begin{equation*}
			\mathbbm{P}(\langle X_t ,1\rangle = 0) > \frac{1}{2},
		\end{equation*}
		where $X$ is the solution to \eqref{eq:SPDE1_sup} with $b_1>0$ and $f\in \mcC_c^+$ as above.
	\end{lemma}
	\begin{proof}
		Without loss of generality we may assume that $\supp(f) \subseteq B(0, K_f)$ and note that $K_f \leq \eta/(2b_1)$.
		Let us introduce the stopping time $\tau$ given by 
		\begin{equation*}
			\tau = \inf\{ t > 0 \colon R_{t-} \geq \eta/(2b_1)\}.
		\end{equation*}
		Then we have the estimate 
		\begin{equation*}
			\mathbb{P}(\tau > t) \geq \mathbbm{P}(R_t < \eta/(2b_1)) = 1- \mathbbm{P}(R_t \geq  \eta/(2b_1)) \geq 1-2b_1\frac{\mathbb{E}[R_t]}{\eta}. 
		\end{equation*}
		Furthermore, for $t \leq 1$ we have by \eqref{eq:gamm_cond} and Proposition~\ref{prop:expectation_RT} that 
		\begin{equation*}
			\frac{2b_1}{\eta}\E[R_t] \leq 	\frac{2b_1}{\eta}\E[R_1] \leq2\cdot 10^{-2},  
		\end{equation*}
		and thus 
		\begin{equation}\label{eq:est_tau_t}
			\bbP(\tau >t) \geq 1-2\cdot10^{-2}, \quad t \in [0,1].
		\end{equation}
		The total mass $z_t = \langle X_t,1 \rangle $ of the solution $X$ to \eqref{eq:SPDE1_sup} satisfies 
		\begin{equation}\label{eq:sde_z_bessel}
			z_t =\langle f, 1\rangle + \int_0^t b_1 \Leb(\{ x\in \R \colon X_s(x) >0 \}) ds + \int_0^t\sqrt{z_s}dB_s, \quad t \geq 0,
		\end{equation}
        where $B$ is a standard Brownian motion. 
		For $s\leq \min(\tau,1)$, we find that 
		\begin{equation*}
			b_1 \Leb(\{ x\in \R \colon X_s(x) >0 \})  \leq b_1 2R_s\leq \eta. 
		\end{equation*}
		Thus, by a comparison argument, we find that almost surely, for $t\leq \min(\tau,1)$, $z_t \leq \tilde{z}_t$, where $\tilde{z}_t$ is the solution of 
		\begin{equation*}
			\tilde{z}_t = \langle f,1\rangle + t\eta + \int_0^t \sqrt{\tilde{z}_s}dB_s, \quad t >0,
		\end{equation*}
        with the same Brownian motion $B$ as in \eqref{eq:sde_z_bessel}. 
		Let us furthermore denote 
		\begin{equation*}
			\sigma_0 = \inf\{t>0 \colon z_t =0\},
		\end{equation*}
		and 
		\begin{equation*}
			\tilde{\sigma}_0 = \inf\{ t >0 \colon \tilde{z}_t = 0\}.
		\end{equation*}
		It is well known (see for instance \cite[(13)-(15)]{Yor03}) that 
        \[
\mathbb P_{m_f}\left(\widetilde\sigma_0\le t\right)
=
\frac{1}{\Gamma(\mu)}
\int_{2m_f/t}^{\infty}
u^{\mu-1}e^{-u}\,du,
\]
and thus since $t\geq \frac{2m_f}{c_\eta}$
we get by \eqref{eq:mu_nuf_cond}
\begin{equation}\label{eq:est_sigma_t}
    \mathbb{P}_{m_f}( \tilde{\sigma}_0 \leq t) \geq \frac{3}{4}.
\end{equation}
		Together \eqref{eq:est_tau_t} and \eqref{eq:est_sigma_t} imply for $t\in [\frac{2m_f}{c_\eta},1]$ that 
		\begin{equation*}
			\begin{split}
				\mathbb{P}(\langle X_t,1 \rangle = 0) &\geq \mathbb{P}(\sigma_0 <t , \tau >t) \\
				& \geq \max\Big( 0, \mathbb{P}(\tilde{\sigma}_0 <t) + \mathbb{P}(\tau >t) - 1\Big)\\
				& \geq \frac{3}{4} + 1-2\cdot10^{-2} -1 > \frac{1}{2}.
			\end{split}
		\end{equation*}
		Since $\{\langle X_t, 1\rangle = 0, \langle X_s, 1\rangle = 0\} = \{\langle X_s, 1\rangle = 0\}$ for $s\leq t$, this finishes the proof.
	\end{proof}
	\begin{proof}[Proof of Theorem~\ref{prop:extinction}]
       If $f \equiv 0$, then the identically zero process is a weak solution of \eqref{eq:SPDE1_sup}, and weak uniqueness implies that $X \equiv 0$ almost surely. Thus, we can assume without loss of generality that $f \neq 0$.
       
		We only prove the claim for $t\in (0,1)$, since the case $t\geq 1$ is an immediate consequence. \\
		Fix some $\eta \in (0,1/2)$ and $t\in (0,1)$.  Since $f$ is compactly supported we can find a decomposition 
		\begin{equation*}
			f = \sum_{i =1}^{k} f_i, 
		\end{equation*}
		for some $k \in \mathbb{N}$ depending only on $K_f = \mathrm{diam}(\supp(f))/2$, $m_f = \langle f, 1\rangle$ and $t$  such that for all $i \in \{1, \ldots, k\}$,
		\begin{enumerate}
			\item $f_i\in \mcC_c^+$, 
			\item $m_ {f_i} = \langle f_i, 1\rangle < \frac{c_\eta}{2}$, with $c_\eta$ from \eqref{eq:mu_nuf_cond},
			\item $\frac{b_1}{\eta}\gamma(1,K_{f_i}, m_{f_i}) \leq 10^{-2}$, where $K_{f_i} = \mathrm{diam}(\supp(f_i))/2$ and $\gamma(1,K_{f_i}, m_{f_i}) = A(1)(K_{f_i}+\sqrt{K_{f_i}} +\beta(1,m_{f_i}))$ with $A$ and $\beta$ from \eqref{eq:expectation_RT_est},
			\item $t \geq \frac{2m_{f_i}}{c_\eta}$.
		\end{enumerate}
		As in Remark~\ref{rem:spde_gen}, one can use standard arguments to show that $X$ is stochastically dominated by $X^{sum} = \sum_{i = 1}^k Z^i$, where $(Z^i)_{i = 1}^k$ are independent and each $Z^i$ is the solution to 
		\begin{equation*}
			d_tZ^i_t(x) = \frac{1}{2}\Delta Z^i_t(x) +b_1\mathbbm{1}_{Z^i_t(x)>0}+ \sqrt{Z^i_t(x)}\dot{W}^i(t,x), \quad Z_0^i = f_i, \quad  x\in \R, t\geq 0.
		\end{equation*}
		Here, $(\dot{W}^i)_{i = 1}^k$ are independent space-time white noises.  Note that we can apply Lemma~\ref{lem:extinction_1} to $Z^i$ for each $i \in \{ 1, \ldots, k\}$ and arrive at 
		\begin{equation*}
			\bbP(\langle X_t, 1\rangle = 0) \geq \prod_{i = 1}^k\bbP(\langle Z^i_t, 1\rangle = 0) \geq \Big(\frac{1}{2}\Big)^k>0,
		\end{equation*}
		which completes the proof.  
	\end{proof}

	\bibliographystyle{plain}
	\bibliography{example.bib}
\end{document}